\documentclass[11pt]{article}
\usepackage[T1]{fontenc}
\usepackage[utf8]{inputenc}
\usepackage{lmodern}
\usepackage[margin=1in]{geometry}
\usepackage{amsmath,amssymb,amsthm,graphicx,booktabs,longtable,array}
\usepackage{xcolor}
\usepackage[numbers,sort&compress]{natbib}
\usepackage[unicode,breaklinks,colorlinks,linkcolor=blue!45!black,citecolor=blue!45!black,urlcolor=blue!45!black]{hyperref}
\hypersetup{pdftitle={Four-arm polyominoes in Golomb's hierarchy},pdfauthor={Angel Ivanov Raychev},pdfsubject={Complete classification and Lean verification}}

\newcommand{\N}{\mathbb{N}}
\newcommand{\Z}{\mathbb{Z}}
\newcommand{\Pshape}[4]{P(#1,#2,#3,#4)}
\newcommand{\Rep}{\mathrm{Rep}}
\newcommand{\R}{\mathrm{R}}
\newcommand{\HS}{\mathrm{HS}}
\newcommand{\BS}{\mathrm{BS}}
\newcommand{\Q}{\mathrm{Q}}
\newcommand{\St}{\mathrm{S}}
\newcommand{\HP}{\mathrm{HP}}
\newcommand{\Plane}{\mathrm{Plane}}
\newcommand{\artifacturl}{https://raw.githubusercontent.com/RaychevAngel/angelraychev.com/ec308bfcb092fa71a78cecd991be6e85e5b369b9/public/polyominoes/lean-proofs.zip}
\newcommand{\paperwebsite}{https://angelraychev.com/four-arm-polyominoes/}
\newtheorem{theorem}{Theorem}[section]
\newtheorem{lemma}[theorem]{Lemma}
\newtheorem{proposition}[theorem]{Proposition}
\newtheorem{corollary}[theorem]{Corollary}
\theoremstyle{definition}

\theoremstyle{remark}

\title{Four-arm polyominoes in Golomb's hierarchy\\[4pt]
\large A complete classification with Lean verification}
\author{Angel Ivanov Raychev\\
\small\href{mailto:angel.ivanov.raychev@gmail.com}{angel.ivanov.raychev@gmail.com}}
\date{September 2026}
\begin{document}
\maketitle
\begin{abstract}
We classify the polyominoes obtained by adjoining four straight arms to a
single square, allowing zero arm lengths, according to their ability to
tile rectangles, half-strips, bent strips, quadrants, strips, half-planes,
and the plane. We also classify their ability to tile an integer enlargement
of themselves. Tiles occupy whole square-grid cells; translations,
rotations, and reflections are permitted. Exactly five capability profiles
occur. For the family $P(n,1,1,0)$, the rectangle profile holds for $n\le3$
and the bent-strip profile, with no half-strip or rep-tiling, for every
$n\ge4$. A cross with four positive arms tiles the plane precisely when
two opposite arms have length one; it never tiles a half-plane. Explicit
periodic constructions and geometric obstructions are combined with finite
symbolic case certificates. A Lean 4 development verifies the full
classification for every natural four-tuple, including the interpretation
of the certificates as statements about arbitrary infinite tilings.
The account incorporates the author's 2020--2021 L- and T-polyomino work,
reconstructs Dahlke's gun argument, and documents the subsequent
AI-assisted proof development and formalization.
\end{abstract}
\begingroup
\small\setlength{\parskip}{0pt}
\tableofcontents
\endgroup
\clearpage
\section{Introduction and provenance}\label{sec:intro}
Golomb's hierarchy distinguishes several increasingly permissive regions
that a single polyomino may tile~\cite{golomb1966}. A plane tiling need
not accommodate a straight boundary, and a strip tiling need not admit
a cap. Determining these distinctions for an infinite parametric family
requires both constructions and obstructions valid without a bound on
the parameters or on the region widths. General questions about the
hierarchy remain a subject of study~\cite{winslow2018}.
Rectangle and half-strip tileability have also been studied systematically
by Reid~\cite{reid1997}.

This paper concerns the union of a horizontal and a vertical row of cells
meeting at a distinguished cell. Zero-length arms include the bar, L, and
T families. We give a single explicit decision rule for all eight tiling
capabilities considered here and prove that rule in Lean. The mathematical
results and their precise grid convention are stated in
Section~\ref{sec:classification}. Some proofs are short geometric arguments;
others use exhaustive finite case analysis with symbolic integer parameters.
The soundness principle for the latter is stated in
Theorem~\ref{thm:certificate-soundness}.
An \href{\paperwebsite}{interactive companion article} gives a tuple
classifier and additional illustrations; the
\href{https://angelraychev.com/polyominoes/paper/}{paper page} collects
the manuscript, sources, citation, and formal artifacts.

\subsection{Earlier work and the present development}
The author's original research was carried out between October 2020 and
March 2021, while at school. The L results were presented in two 2021
conference papers~\cite{raychevQuadrant2021,raychevHalfPlane2021}.
The author also obtained the T classification during that period, apart
from $\Pshape4110$ and $\Pshape5110$, in handwritten work that was not
prepared as a digital publication. Those handwritten T proofs were not
provided to the AI assistant used for the present development. The
reconstructed T results agree with the author's earlier classification.
The dates of the research are distinct from the later presentation dates.

Dahlke's \emph{Gun Theorem} is a direct source for the argument concerning
$\Pshape n110$~\cite{dahlkeGun}. Its section on a one-cell face treats
every $n\ge4$ and uses three supported-corner configurations to rule out
rectangles. Our formal corner system makes the reflected configurations
explicit and proves additional progress and containment inequalities for
half-strips and rep-tilings. We do not attribute the underlying
corner-propagation method, or rectangle nonrectifiability of the two small
cases, to the present investigation. The theorem was documented in the
revision of Tulleken's book dated 19 May 2019~\cite{tulleken2019}; the
original composition date of Dahlke's online argument is unspecified.
The half-strip and rep-tile conclusions are established here with explicit
certificates and formal proofs. We make no claim that these conclusions
could not also be extracted from the earlier argument.

The rectangle tiling of the Y hexomino used below is a historical
construction, attributed to Marlow and Dahlke in Reid's
record~\cite{reidY,marlow1985,dahlke1989y}. It was converted to cell coordinates and
independently verified; its discovery and minimality are not contributions
of this paper. The L classification is likewise credited to the original
papers rather than to its formal reconstruction.
Dahlke's separate \emph{L Theorem}, completed in its repository in
September 2020, studies rectangle tileability for a broader family of
thick L shapes~\cite{dahlkeL}. That related result is distinct in scope
from the half-plane classification used here.

The present contribution is a unified, explicit classification with
constructions, complete obstruction certificates, and a formal proof of
every capability for every tuple. It also documents the genuine-cross
analysis developed during the subsequent investigation. This contribution
statement does not assert priority for every individual tiling or local
lemma, and does not identify the number of generated proof obligations with
the intrinsic difficulty of the underlying mathematics.

\subsection{Use of generative AI}
OpenAI's Astra 6 model, operating through the Codex harness, was used
extensively in the development of this paper and its accompanying
artifacts. Astra 6's role included mathematical exploration,
reconstruction and development of proofs, discovery of constructions,
counterexample searches, certificate generation, Lean formalization,
drafting, and preparation of figures and source packages. In particular,
the work completing the two exceptional T cases, the genuine-cross
analysis, and the formal development were produced through this
AI-assisted investigation initiated and directed by the author. This
description distinguishes those contributions from the author's earlier
work and from the cited literature. AI assistance was substantive and was
not confined to language editing. Lean verification and independent
certificate checks concern the stated mathematical assertions; they do
not themselves establish historical priority or the quality of the
exposition. The author is responsible for the submitted manuscript.

\section{Definitions and the complete classification}\label{sec:classification}
Throughout, $\N=\{0,1,2,\ldots\}$. Identify the cell
$[x,x+1]\times[y,y+1]$ with $(x,y)\in\Z^2$; disjointness below means
disjoint cell sets, so neighboring closed squares may share edges.
For $a,b,c,d\in\N$ define
\begin{equation}\label{eq:shape}
\Pshape abcd=\{(x,0):-c\le x\le a\}\cup
                \{(0,y):-d\le y\le b\}.
\end{equation}
The arm parameters run east, north, west, and south and exclude the
junction. Thus the area is $a+b+c+d+1$. A copy is an integer translate of
any of the eight square-grid rotations or reflections. A tiling of a
region is an arbitrary collection of copies contained in the region,
pairwise disjoint, whose union is the region. No periodicity assumption
is imposed on an infinite tiling.

\begin{figure}[htbp]\centering
\includegraphics[width=.47\linewidth]{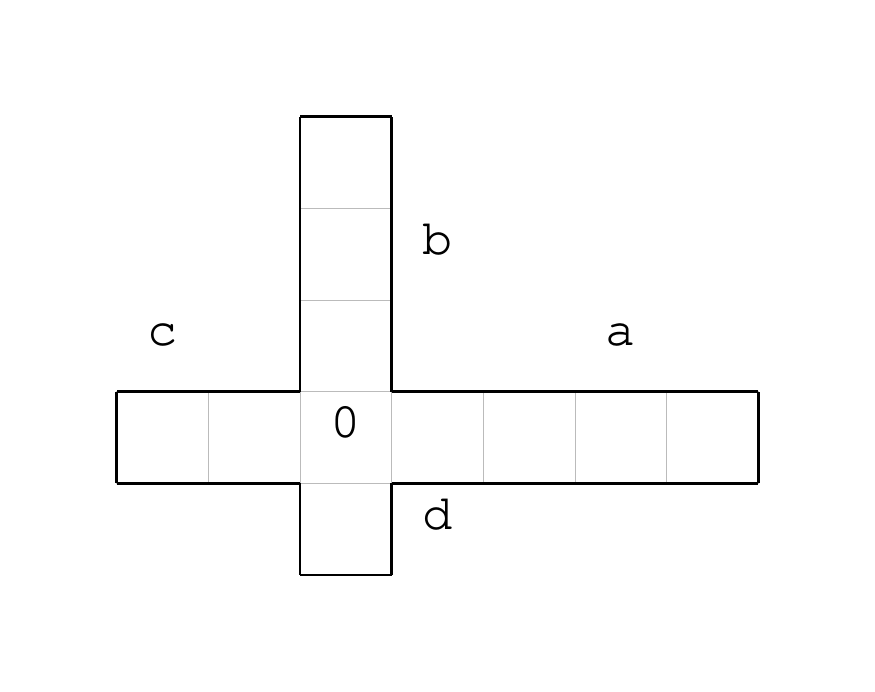}
\caption{The convention for $\Pshape4321$. The junction is counted once.}
\label{fig:shape}
\end{figure}

For positive integers $u,v,w$, the region capabilities are
\begin{center}\small
\begin{tabular}{@{}ll@{}}\toprule
Symbol&Some region of the following form is tileable\\\midrule
$\R$&$\{0,\ldots,u-1\}\times\{0,\ldots,v-1\}$\\
$\HS$&$\N\times\{0,\ldots,w-1\}$\\
$\BS$&$\{(x,y)\in\N^2:x<u\text{ or }y<v\}$\\
$\Q$&$\N^2$\\
$\St$&$\Z\times\{0,\ldots,w-1\}$\\
$\HP$&$\Z\times\N$\\
$\Plane$&$\Z^2$\\\bottomrule
\end{tabular}
\end{center}
Every width quantifier is existential and ranges over all positive
integers. A failure of $\HS$, for example, excludes every such width.
The capability $\Rep$ means that copies tile an enlargement of the
polyomino by a natural integer linear factor $k\ge2$: every original cell
is replaced by its $k\times k$ block. This is the lattice rep-tile
convention. Off-grid placements and dissections into unequal-scale copies
are outside the statement.

The region implications are
\begin{equation}\label{eq:hierarchy}
\R\Longrightarrow\HS\Longrightarrow\BS
\Longrightarrow\begin{cases}\Q\\\St\end{cases}
\Longrightarrow\HP\Longrightarrow\Plane,
\qquad \R\Longrightarrow\Rep\Longrightarrow\Q.
\end{equation}
They do not form a single chain. Define the following \emph{profiles},
listing all positive capabilities and excluding every unlisted one:
\begin{center}\small
\begin{tabular}{@{}lcccccccc@{}}\toprule
Profile&$\R$&$\Rep$&$\HS$&$\BS$&$\Q$&$\St$&$\HP$&$\Plane$\\\midrule
Rectangle&yes&yes&yes&yes&yes&yes&yes&yes\\
Bent strip&--&--&--&yes&yes&yes&yes&yes\\
Strip&--&--&--&--&--&yes&yes&yes\\
Plane only&--&--&--&--&--&--&--&yes\\
Non-tiler&--&--&--&--&--&--&--&--\\\bottomrule
\end{tabular}
\end{center}

\noindent\begin{minipage}{\linewidth}
\begin{theorem}[Complete classification]\label{thm:classification}
Every polyomino in \eqref{eq:shape} has exactly the profile given below.
For an L shape use $\Pshape ab00$ with $a\ge b\ge1$; for a T shape use
$\Pshape abc0$ with $a\ge c\ge1$ and $b\ge1$. These normal forms are
obtained by square-grid symmetries. Bars include the single cell and all
representations with only two opposite positive arms.
\begin{center}\small\normalfont
\begin{tabular}{@{}p{.72\linewidth}l@{}}\toprule
Normalized shape and parameter condition&Profile\\\midrule
Bar&Rectangle\\
L: $b=1$&Rectangle\\
L: $b=2$ or $(a,b)=(4,3)$&Strip\\
L: all other $a\ge b\ge1$&Plane only\\\addlinespace
T: $b=c=1$, $a\le3$&Rectangle\\
T: $b=c=1$, $a\ge4$&Bent strip\\
T: $b=1$, $c\ge2$&Strip\\
T: $b\ge2$ and at least one of $c=1$, $b=2$, or $(c=2,\ b=a+3)$&Plane only\\
T: all remaining parameters&Non-tiler\\\addlinespace
Cross: $a,b,c,d\ge1$, with $a=c=1$ or $b=d=1$&Plane only\\
Cross: all remaining positive parameters&Non-tiler\\\bottomrule
\end{tabular}
\end{center}
\end{theorem}
\end{minipage}

\begin{corollary}\label{cor:gun}
For every $n\in\N$, $\Pshape n110$ has the rectangle profile if $n\le3$
and the bent-strip profile if $n\ge4$.
\end{corollary}
\begin{corollary}\label{cor:planecriteria}
A genuine T in normal form tiles the plane if and only if
$b\le2$, or $c=1$, or $(c=2\text{ and }b=a+3)$.
A genuine cross tiles the plane if and only if two opposite arms have
length one. In particular, four arms all greater than one give a non-tiler.
\end{corollary}
The five profiles are exhaustive for this family, not for arbitrary
polyominoes. The proofs below establish the positive constructions and
the separating obstructions; \eqref{eq:hierarchy} then determines each
remaining capability. The normalization of every natural tuple, including
duplicate descriptions of bars, is also part of the formal theorem.

\section{Hierarchy implications and finite certificates}\label{sec:foundations}
\begin{proposition}\label{prop:hierarchy}
All implications in \eqref{eq:hierarchy} hold for the lattice convention
of Section~\ref{sec:classification}.
\end{proposition}
\begin{proof}
Repeated rectangles give a half-strip. A half-strip of width $w$ and
a rotated copy beginning at $(0,w)$ give a bent strip with widths $w,w$.
For a bent strip of widths $u,v$, its translates by $k(u,v)$,
$k\in\N$, partition the quadrant: the unique layer containing $(x,y)$
has index $\min(\lfloor x/u\rfloor,\lfloor y/v\rfloor)$.
A full strip gives a half-plane by parallel repetition.

For the remaining infinite-region implications use finite local choice.
Only finitely many copies of a fixed finite tile contain a specified
cell. Translate a sequence of tiled regions so that increasingly large
finite patches of the target region are covered and its required boundary
is respected. Successively restrict to subsequences that agree on every
larger patch. The consistent limit consists of whole copies, remains
disjoint, and covers every target cell. Moving far down one arm of a
bent strip gives a strip; moving far along the boundary of a quadrant
gives a half-plane; moving into the interior of a half-plane gives the
plane. This argument asserts existence of a limiting tiling and does
not assert periodicity.

If an $r\times s$ rectangle is tileable, choose an integer $k>1$ divisible
by both $r$ and $s$. Every $k\times k$ block is tiled by such rectangles,
so replacing each cell of the prototile by a block proves $\Rep$.
Conversely, iterate a rep-tiling and place a fixed convex corner of each
successively enlarged tile at the origin. For any prescribed bounded
quadrant patch, sufficiently large enlargements contain the patch and
have the two required adjacent boundary segments. Since tile diameters
are bounded, a copy meeting the prescribed patch cannot reach past
either of those increasingly long boundary segments. The same finite-choice
argument gives a quadrant tiling.
\end{proof}

\subsection{What a negative certificate proves}
A finite search box is not itself an obstruction to an infinite tiling:
a tile may extend beyond the box. Our certificates instead query cells
of a hypothetical actual tiling and retain every cell of each selected
copy. The arithmetic parameters may be fixed or range over an explicitly
specified unbounded domain.

\begin{theorem}[Soundness of finite obstruction trees]\label{thm:certificate-soundness}
Fix a parameter domain, a region, and a finite collection of whole
selected tiles required to occur in a hypothetical exact tiling. Suppose
there is a finite rooted tree with the following properties for every
parameter choice in the domain. Each nonterminal node specifies a cell
of the region not covered by its selected tiles. Its children exhaust
every congruent whole tile through that cell which is contained in the
region and disjoint from the selected tiles. Each child adds that tile.
Every terminal node has a proved contradiction, either no possible cover
or a local obstruction whose hypotheses hold for the selected tiles.
Then no exact tiling contains the root configuration.
\end{theorem}
\begin{proof}
Assume such a tiling exists. At a nonterminal node its queried cell has
a unique covering tile in the tiling. That tile satisfies the containment
and disjointness conditions, so exhaustiveness gives a child whose
selected tiles also belong to the tiling. Induction on the finite tree,
starting at its leaves, contradicts the terminal obstruction. For symbolic
parameters fix an arbitrary admissible tuple first; the same induction
then applies because each branch implication was proved throughout the
domain.
\end{proof}

An analogous certificate may terminate in a specified successor
configuration instead of a contradiction. The same induction proves the
local transition statement. Section~\ref{sec:guns} combines such
transitions with a well-founded rank. This distinction is necessary:
a finite transition certificate is not by itself an infinite-region
obstruction.

The certificates enumerate all eight orientations and allow an arbitrary
integer junction for the covering tile. A proposed child can be impossible
for some parameter values; that branch is then vacuous, not an omitted
case. Search and external solvers propose the trees and arithmetic claims.
The formal development checks the arithmetic and proves its connection
to copies, containment, coverage, and disjointness in an actual tiling.
The finite symbolic trees used later are explicit parts of the proof,
available in the versioned artifacts. A named computer-assisted proposition
below therefore refers to a supplied exhaustive proof, not to unsuccessful
testing over finitely many parameter values.

\section{Explicit positive constructions}\label{sec:constructions}
A lattice construction below specifies finitely many copies and a rank-two
translation lattice. It proves a plane tiling when the cells of the copies
represent each class of $\Z^2/\Lambda$ exactly once. This verifies both
coverage and nonoverlap, not just agreement of areas. Write
$\sigma(x,y)=(y,x)$ and $-T=\{(-x,-y):(x,y)\in T\}$.

\subsection{Bars and small rectangle cases}
A bar is already a rectangle. For $\Pshape1110$ and $\Pshape2110$ the
following partitions give rectangles; each letter denotes one complete
copy. Checking the cells of each letter gives the required congruences.
\begin{center}
\begin{tabular}{cc}
$\Pshape1110$&$\Pshape2110$\\[3pt]
\begin{tabular}{@{}l@{}}\ttfamily BCCC\\\ttfamily BBCD\\\ttfamily BADD\\\ttfamily AAAD\end{tabular}&
\begin{tabular}{@{}l@{}}\ttfamily BBBBEHHHHJ\\\ttfamily ABEEEEHGJJ\\\ttfamily ADDDDGGGGJ\\\ttfamily AACDFFFFIJ\\\ttfamily ACCCCFIIII\end{tabular}
\end{tabular}
\end{center}
The historical $24\times23$ rectangle of 92 Y hexominoes provides the
$\Pshape3110$ case~\cite{reidY,dahlke1989y}; see Figure~\ref{fig:yrect}.
The formal artifact checks all 552 cells exactly once against the
explicit 92-copy witness. Only existence is needed for this classification.
\begin{figure}[htbp]\centering
\includegraphics[width=.64\linewidth]{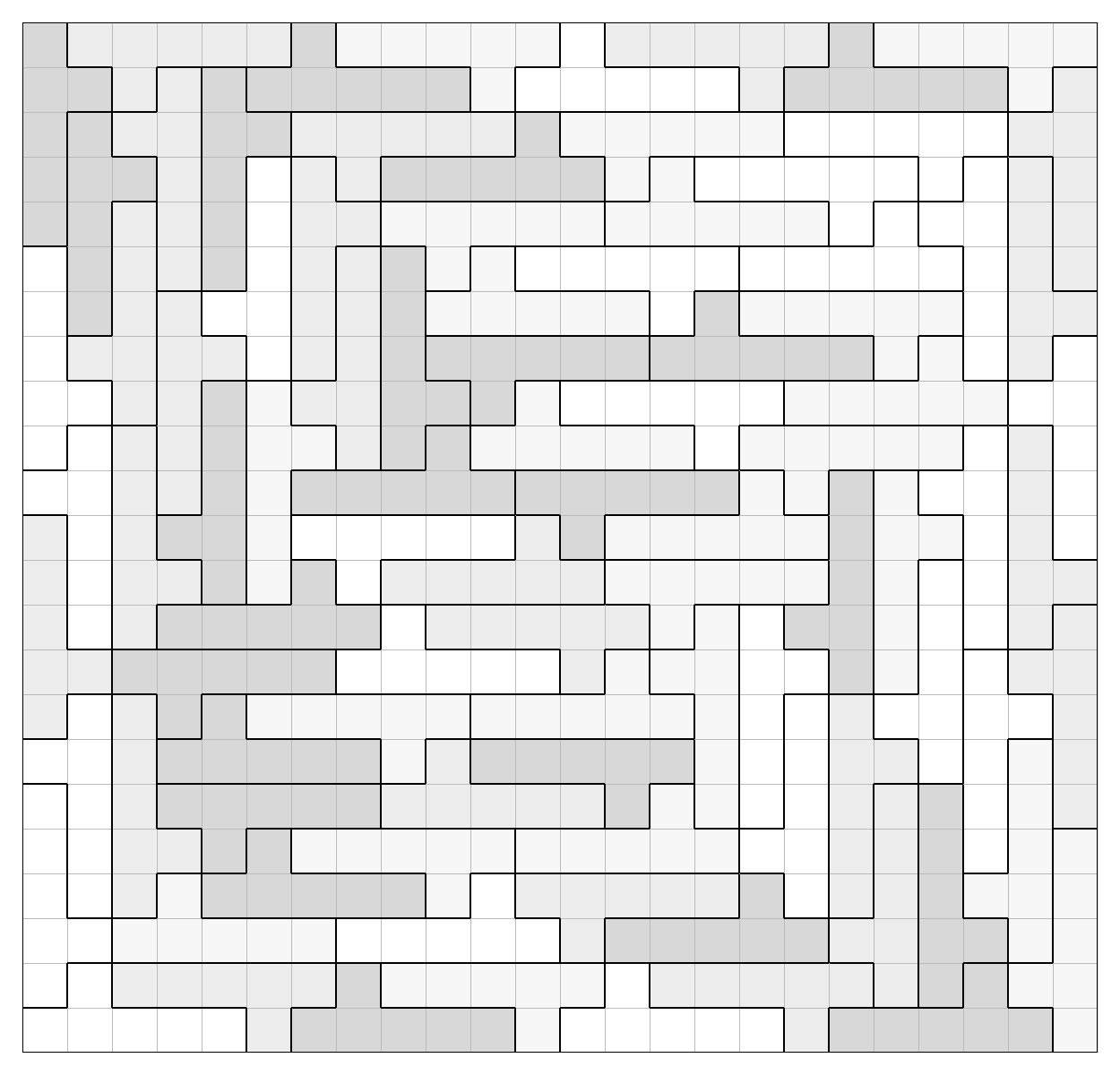}
\caption{A cell-coordinate reconstruction of the Y-hexomino rectangle
recorded by Reid and attributed to Marlow and Dahlke. Each region is one
copy of $\Pshape3110$.}
\label{fig:yrect}
\end{figure}

\subsection{Strips and bent strips for T shapes}
\begin{proposition}\label{prop:tstrip}
Every $T=\Pshape a1c0$ with $a,c\ge1$ tiles a strip of width two.
If $c=1$, it also tiles a bent strip of widths two and two.
\end{proposition}
\begin{proof}
Put $L=a+c+1$ and $m=L+1$. For each $j\in\Z$ place
\[
T+(jm+c,0),\qquad -T+(jm+L,1).
\]
On row zero the first tile covers $jm,\ldots,jm+L-1$ and the second
the cell $jm+L$. On row one the first contributes $jm+c$ and the second
$jm+c+1,\ldots,jm+c+L$. Each row is a disjoint partition into successive
blocks of length $m$; every tile remains in the two rows.

For $c=1$ restrict to $j\ge0$. The tiled region is
\[
J=\{(x,0):x\ge0\}\cup\{(x,1):x\ge1\}.
\]
Its reflected translate $\sigma J+(0,1)$ is
$\{(0,y):y\ge1\}\cup\{(1,y):y\ge2\}$. These regions are disjoint
and their union is $\{(x,y)\in\N^2:x<2\text{ or }y<2\}$.
\end{proof}
\begin{figure}[htbp]\centering
\includegraphics[width=.92\linewidth]{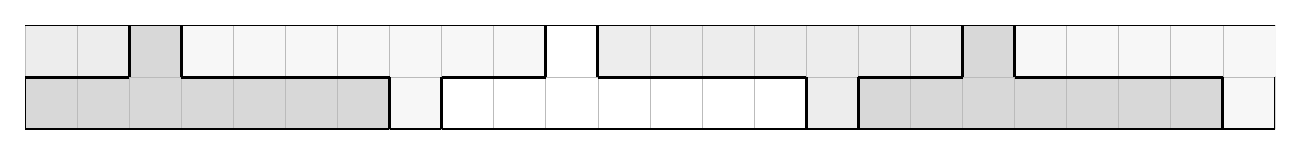}\\[8pt]
\includegraphics[width=.64\linewidth]{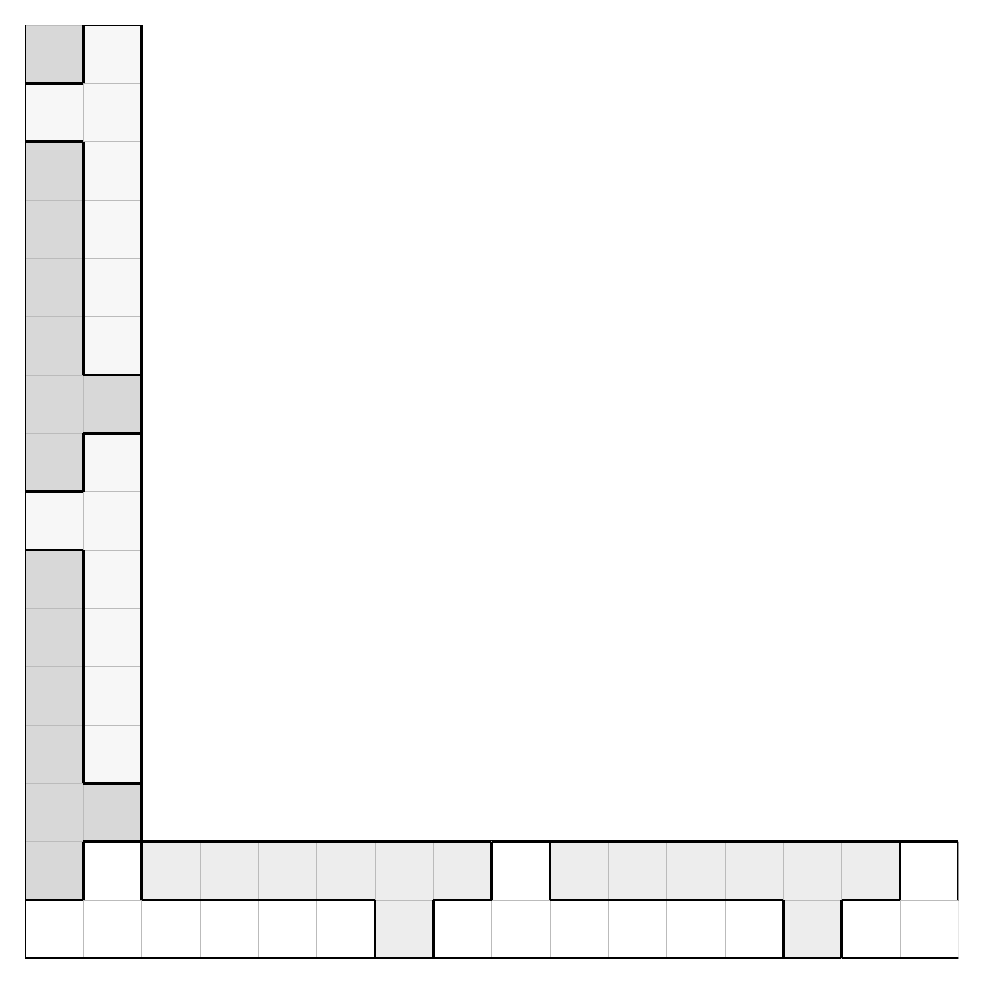}
\caption{A strip for $\Pshape4120$ (top) and a bent strip for
$\Pshape4110$ (bottom). The drawings are cropped; the cut ends do not
assert finite rectangle tilings.}
\label{fig:strips}
\end{figure}

\subsection{Three plane constructions for T shapes}
\begin{proposition}\label{prop:tplane}
A genuine T in normal form tiles the plane whenever $c=1$, $b=2$, or
$(c=2,\ b=a+3)$.
\end{proposition}
\begin{proof}
\emph{A unit crossbar arm.} Set $T=\Pshape ab10$, $m=a+b+2$, and
$\Lambda=\langle(2m,0),(2,1)\rangle$. Repeat the two copies
$T$ and $-T+(2a+1,0)$ over $\Lambda$. The quotient coordinate is
$x-2y\pmod{2m}$. The two horizontal bars supply $-1,\ldots,2a+2$.
The first stem supplies the remaining even residues
$2a+4,2a+6,\ldots,2m-2$; the second supplies the remaining odd residues
$2a+3,2a+5,\ldots,2m-3$. These sets partition all $2m$ residues.

\emph{A two-cell stem.} Set $T=\Pshape a2c0$, $m=a+c+3$, and
$\Lambda=\langle(m,0),(c+1,2)\rangle$. Repeat $T$ and
$-T+(a+1,1)$. Use quotient coordinates
\[
j=y\bmod2,\qquad u=x-\lfloor y/2\rfloor(c+1)\pmod m.
\]
For $j=0$ the first bar supplies $[-c,a]$, its upper stem endpoint
supplies $-c-1$, and the second stem supplies $a+1$. These are $m$
consecutive residues. For $j=1$ the second bar supplies $1,\ldots,m-2$;
the two remaining stem cells supply $0,m-1$.

\emph{The four-copy family.} Set $b=a+3\ge5$ and $T=\Pshape ab20$.
Repeat the four copies
\begin{equation}\label{eq:fourcopy}
T,\quad -\sigma T+(b-2,-2),\quad -T+(b-1,1),\quad \sigma T+(1,3)
\end{equation}
over $\Lambda=\langle(4,4),(b,-b)\rangle$. Each tile has $2b$ cells,
and the lattice index is $8b$. Put
$d=x-y$, $q=\lfloor d/(2b)\rfloor$, $\delta=d-2bq$, and
$\eta=y+bq\pmod4$. These are coordinates on the quotient.
Dividing each tile into its bar (including the junction) and its stem
(excluding the junction) gives the following four representatives at
each $\delta$; the entries are $y$ coordinates after reducing $d$ to
$\delta$, before reduction modulo four.
\begin{center}\small
\begin{tabular}{@{}ll@{}}\toprule
$\delta$&Four representatives\\\midrule
$0$&$0,-2,3,1$\\
$1\le\delta\le b-3$&$0,-2,1,3$\\
$b-2$&$-2,1,3,0$\\
$b-1$&$-2,1,-1,0$\\
$b$&$1,0,-2,-1$\\
$b+1\le\delta\le2b-3$&$b-\delta,b-2-\delta,b-1-\delta,b+1-\delta$\\
$2b-2$&$-b,2-b,1-b,3-b$\\
$2b-1$&$-b,3-b,1-b,2-b$\\\bottomrule
\end{tabular}
\end{center}
Every row is a complete set modulo four; the listed intervals exhaust
$0\le\delta<2b$. This verifies \eqref{eq:fourcopy} for all $a\ge2$.
\end{proof}
\begin{figure}[htbp]\centering
\includegraphics[width=.70\linewidth]{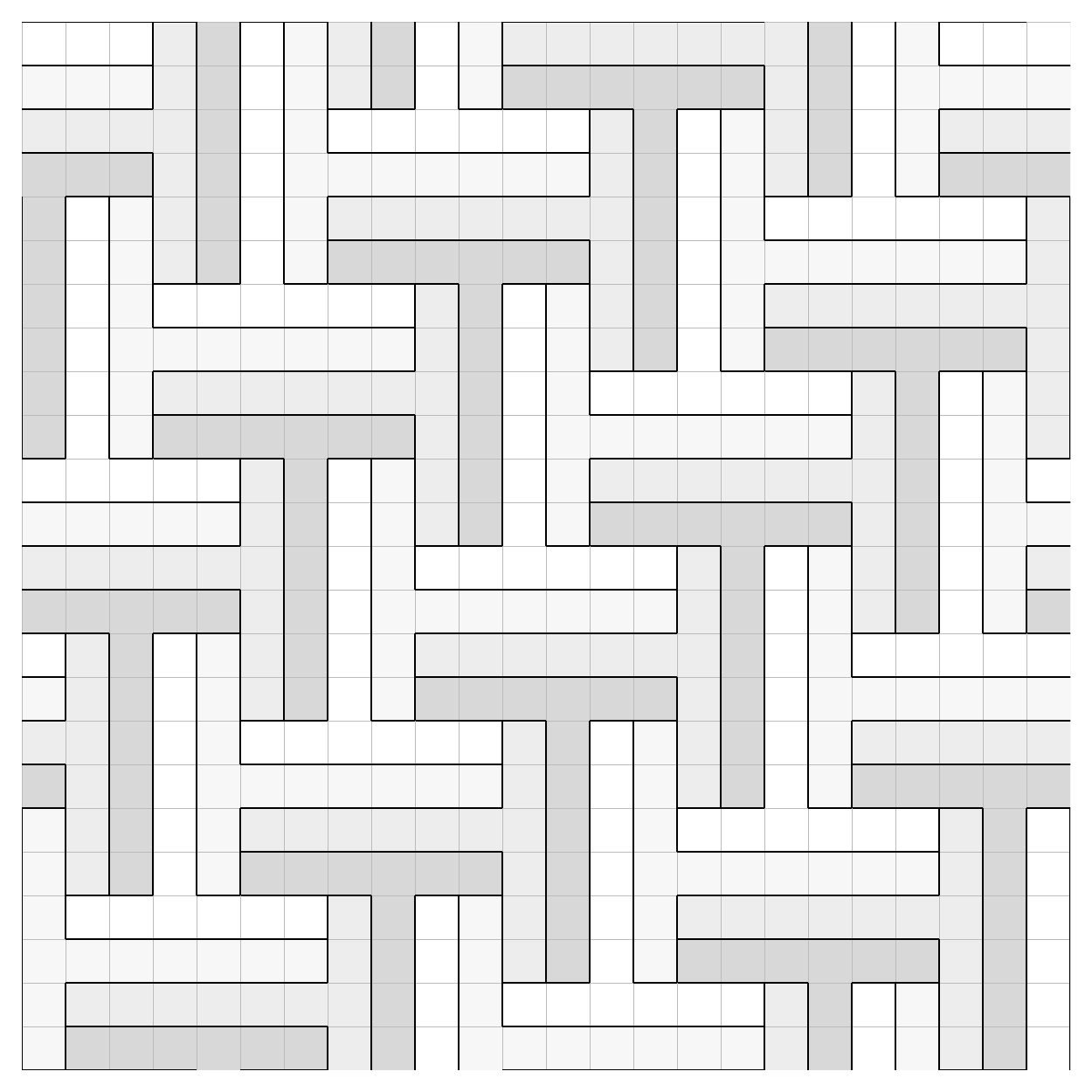}
\caption{The four-copy plane construction for $\Pshape3620$. Boundary
pieces are cropped by the drawing window.}
\label{fig:fourcopy}
\end{figure}

\subsection{Crosses with opposite unit arms}
\begin{proposition}\label{prop:crosspositive}
For $a,c\ge1$, $C=\Pshape a1c1$ tiles the plane.
\end{proposition}
\begin{proof}
Put $m=a+c+3$ and repeat $C$ and $-C+(a+2,1)$ over
$\Lambda=\langle(m,0),(1,2)\rangle$.
Use quotient coordinates $j=y\bmod2$ and
$u=x-(y-j)/2\pmod m$.
For even parity the first bar supplies $-c,\ldots,a$, and the second
tile's tips supply $a+1,a+2$. For odd parity the second bar supplies
$2,\ldots,m-1$, and the first tile's tips supply $0,1$. Both parity
classes are partitioned exactly once.
\end{proof}

\subsection{L constructions}\label{sec:lpositive}
The positive L families are those in the author's half-plane
classification~\cite{raychevHalfPlane2021}. Here are explicit witnesses
in the arm convention.

For $L=\Pshape a100$, the copies $L$ and $-L+(a+1,1)$ partition an
$(a+2)\times2$ rectangle. Every $L=\Pshape ab00$ tiles the plane:
under $x-y\pmod{a+b+1}$ its cells give the consecutive residues
$-b,\ldots,a$. Thus translates over
$\langle(a+b+1,0),(1,1)\rangle$ partition the grid.

For $L=\Pshape a200$, $a\ge2$, the following six tiles, repeated by
$(6j,0)$ for $j\in\Z$, tile the strip of height $a+3$. Each listed bar
and foot includes the junction, which is counted once.
\begin{center}\small
\begin{tabular}{@{}cll@{}}\toprule
Tile&Vertical bar&Horizontal foot\\\midrule
A&$x=0,\ 0\le y\le a$&$y=0,\ 0\le x\le2$\\
B&$x=1,\ 1\le y\le a+1$&$y=a+1,\ -1\le x\le1$\\
C&$x=2,\ 2\le y\le a+2$&$y=a+2,\ 0\le x\le2$\\
D&$x=3,\ 2\le y\le a+2$&$y=a+2,\ 3\le x\le5$\\
E&$x=4,\ 1\le y\le a+1$&$y=1,\ 2\le x\le4$\\
F&$x=5,\ 0\le y\le a$&$y=0,\ 3\le x\le5$\\\bottomrule
\end{tabular}
\end{center}
Rows $2,\ldots,a$ receive one cell from each vertical bar. Row zero
is covered by the feet of A and F. Row one receives A, B, the foot of E,
and F, covering residues $0,1,2,3,4,5$. Row $a+1$ receives the foot
of B at residues $5,0,1$ and C,D,E at $2,3,4$; row $a+2$ receives
the feet of C and D. This proves exact coverage in every row.

For the remaining case $\Pshape4300$, the following placements give a
strip of height eight and horizontal period eight. A placement is
written as its junction and its directed arms.
\begin{center}\small
\begin{tabular}{@{}cc@{}}\toprule
Junction&$(E,N,W,S)$\\\midrule
$(4,1)$&$(0,3,4,0)$\\
$(3,2)$&$(0,4,3,0)$\\
$(5,0)$&$(0,4,3,0)$\\
$(2,7)$&$(3,0,0,4)$\\
$(7,5)$&$(0,0,3,4)$\\
$(8,6)$&$(0,0,4,3)$\\
$(6,0)$&$(3,4,0,0)$\\
$(9,7)$&$(0,0,3,4)$\\\bottomrule
\end{tabular}
\end{center}
Reduction modulo eight in the horizontal coordinate gives each of the
64 strip cells exactly once. The motif is also checked in Lean.
The drawing in the original half-plane paper can be reconstructed as a
width-eight construction although its text calls the width six. This
dimensional correction leaves the existence classification unchanged;
it does not assert impossibility of every width-six strip.
\begin{figure}[htbp]\centering
\includegraphics[width=.92\linewidth]{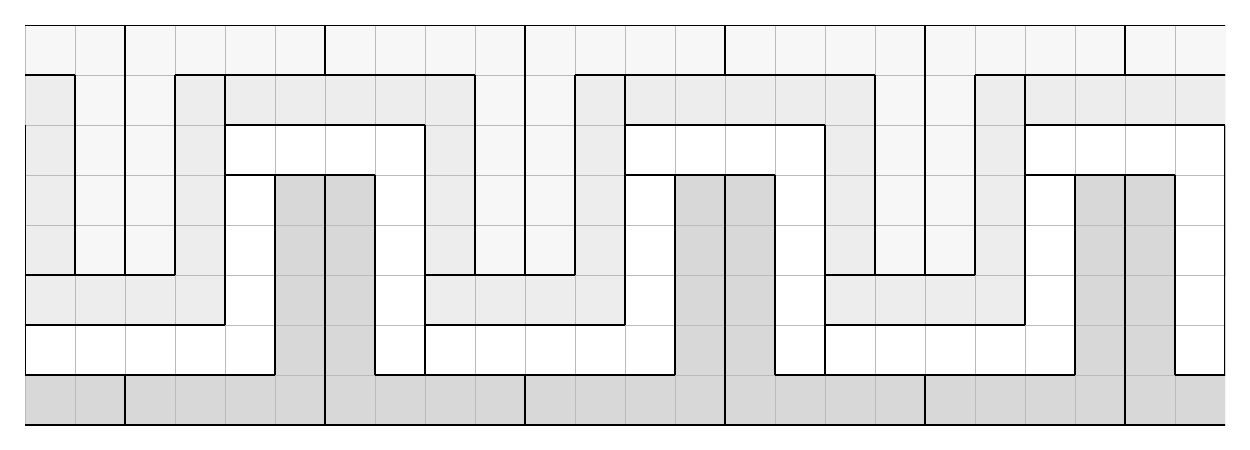}
\caption{Three periods of the width-eight strip for $\Pshape4300$,
whose bounding dimensions are five by four.}
\label{fig:lstrip}
\end{figure}

\section{The one-cell-stem family and Dahlke's corner argument}\label{sec:guns}
The positive bent-strip construction for every $\Pshape a110$ was
given in Proposition~\ref{prop:tstrip}. This section proves the half-strip
and rep-tile obstructions for every $a\ge4$.
The three basic corner configurations come directly from section 23 of
Dahlke's gun argument~\cite{dahlkeGun}. We reconstruct its local cases
as complete certificates and add explicit inequalities suitable for the
two infinite or scaled-region deductions below. The rectangle obstruction
is therefore credited to the earlier argument, including $a=4,5$.

\subsection{A local transition statement}
A \emph{supported corner} at $(X,Y)$ consists of a selected collection
of whole tiles and occupied boundary cells. Relative to $(X,Y)$ its wall
contains $(-1,0),\ldots,(-1,w-1)$ and its floor contains
$(0,-1),\ldots,(f-1,-1)$. In the northeast quadrant the selected
cells are exactly the decoration $E$ specified below. Cells of a
selected tile outside this quadrant remain part of that tile; they are
not cut away. Initially the wall and floor can be supplied by the
exterior of the target region. Interior successors use actual tile cells.
\begin{center}\small
\begin{tabular}{@{}ccclcc@{}}\toprule
State $s$&$w$&$f$&Decoration $E$&Rank $r(s)$&Potential $h(s)$\\\midrule
$\alpha$&4&4&$\varnothing$&0&0\\
$\beta$&5&6&$\{(0,0)\}$&1&2\\
$\gamma$&5&6&$\{(0,0),(1,0)\}$&2&2\\
$\beta^{\mathsf T}$&6&5&$\{(0,0)\}$&1&1\\
$\gamma^{\mathsf T}$&6&5&$\{(0,0),(0,1)\}$&0&0\\\bottomrule
\end{tabular}
\end{center}
The superscript denotes reflection in $x=y$.
\begin{figure}[htbp]\centering
\includegraphics[width=\linewidth]{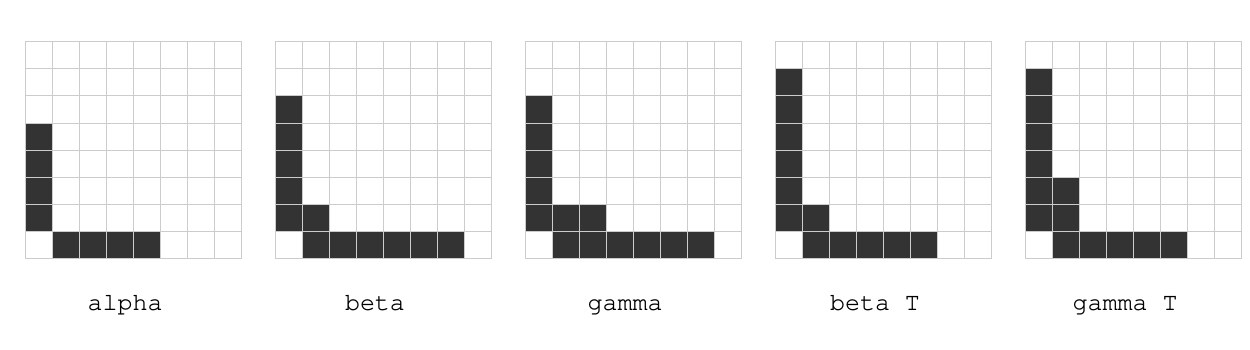}
\caption{The five supported-corner types. Dark cells are occupied.
The two reflected types make the orientation bookkeeping explicit.}
\label{fig:corners}
\end{figure}

\begin{proposition}[Certified corner transition]\label{prop:transition}
For $a\ge4$, a supported corner of type $s$ in an exact tiling by
$\Pshape a110$ has a supported successor of type $s'$ obtained by
adding finitely many whole tiles, at displacement $(\Delta x,\Delta y)$,
provided the supported queries lie in the target region and the covering
tiles remain below its specified ceiling. The selected and supporting
cells must lie below that ceiling; whenever a selected cell belongs to
a tile, that whole tile is selected. The certificates query only
relative columns $0,\ldots,7$ and prove
\begin{align}
&\Delta x,\Delta y\ge0,\qquad \Delta x+\Delta y>0,\label{eq:progress}\\
&s=\alpha\ \Longrightarrow\ \Delta x>0\text{ and }\Delta y>0,\nonumber\\
&\Delta y=0\ \Longrightarrow\ r(s')<r(s),\label{eq:zerorank}\\
&\Delta x-4\Delta y\le h(s)-h(s').\label{eq:potential}
\end{align}
The new selected cells meet the successor quadrant in exactly its listed
decoration. In particular, they supply the successor wall and floor.
\end{proposition}
\begin{proof}[Computer-assisted proof]
The local statement is supplied separately by finite certificates for
$a=4$ and $a=5$ and by one symbolic family for all $a\ge6$.
At each node all whole copies through the queried cell are exhausted;
overlaps are rejected and all surviving terminal nodes give a specified
successor. Finite induction as in Theorem~\ref{thm:certificate-soundness}
proves the transition. The certificates for four and five contain 1750
and 872 nodes respectively; the unbounded family is proved using 8309
arithmetic lemmas and five geometric transition replays.
The zero-height transitions can only follow
$\gamma\to\beta^{\mathsf T}\to\alpha$ or
$\beta\to\alpha$; some are absent for $a=4$.
All four inequalities and the validity of each supported query are
checked in the formal replay. The formal inputs and exact whole-tiling
interpretation are given in Section~\ref{sec:formalization}.
More precisely, a query's height is bounded by that of an already occupied
support cell, or by the top of an already selected whole tile. These
bounds are part of the replay. If the known supports lie below a ceiling,
the query therefore lies below it \emph{before} its covering tile is
requested. Containment of that newly selected tile preserves the same
invariant for subsequent queries. This is the supported-query condition
used in both applications below.
\end{proof}

These conditions state more than the existence of an endless sequence
of corners. For a rectangle, unbounded $X+Y$ would suffice for a
contradiction. In a half-strip, horizontal escape must also be excluded;
\eqref{eq:zerorank} supplies that additional information. For an enlarged
T, \eqref{eq:potential} controls the distance from the stem.

\begin{theorem}\label{thm:gunhs}
For every $a\ge4$, $\Pshape a110$ tiles no half-strip and no rectangle.
\end{theorem}
\begin{proof}
If a half-strip of positive height $H$ were tileable, stack four copies
to obtain one of height $K=4H\ge4$. At its initial corner the exterior
boundary supplies state $\alpha$. At any subsequent corner the queried
cells are in the half-strip and whole tiles stay below its ceiling;
Proposition~\ref{prop:transition} therefore applies.
For a corner of type $s$ at height $Y$ assign the integer
\[
                         3(K-Y)+r(s).
\]
It is nonnegative because the supported corner lies below the ceiling.
If a transition has $\Delta y=0$, \eqref{eq:zerorank} decreases the
rank. If $\Delta y\ge1$, the first term decreases by at least three
while $r$ can increase by at most two. Each transition therefore strictly
decreases a nonnegative integer, a contradiction. A rectangle would give
a half-strip by repetition.
\end{proof}

\begin{theorem}\label{thm:gunrep}
For every $a\ge4$, $\Pshape a110$ is not a lattice rep-tile.
\end{theorem}
\begin{proof}
Iterating a hypothetical rep-tiling provides integer scales as large as
needed; choose $k\ge a+2$. Reflect the enlarged T and measure coordinates
leftward and upward from the right-hand end of its long arm. Its cell set
becomes
\[
\{0\le X<(a+2)k,\ 0\le Y<k\}\ \cup\
\{ak\le X<(a+1)k,\ k\le Y<2k\}.
\]
Before $X=ak$ the region has height $k$. Start the $\alpha$ certificate
at $(0,0)$. A tile through a query extends at most $a+1$ cells horizontally
from that query, since the largest difference of two horizontal cell
coordinates in any copy is $a+1$. Thus the initial queried tiles have
$X\le a+8<ak$ and remain in the rectangular arm. They consequently
stay below height $k$, so the local transition hypotheses hold. Its first
successor has both coordinates positive.

Summing \eqref{eq:potential} gives the invariant
$X-4Y\le-h(s)$. For an interior corner still before the stem, its
actual vertical wall of height $w\ge4$ lies in that arm, so $Y+w\le k$.
Consequently
\begin{equation}\label{eq:stemclearance}
ak-X\ge(a-4)k+4w+h(s)>a+8.
\end{equation}
For the strict inequality, subtract $a+8$ from the middle expression:
\[
(a-4)(k-1)+4(w-4)+h(s)+4\ge4.
\]
Every whole tile through the next query therefore remains before the
stem, since the query is at most seven columns to the right of the
corner. Coverage and containment validate all supported queries below
the height-$k$ ceiling. The successor floor is supplied by these selected
cells, so its origin also remains before the stem. This proves the
invariant and clearance inductively, without extending the rectangular
arm across the concave boundary. The same rank
$3(k-Y)+r(s)$ as in Theorem~\ref{thm:gunhs} now strictly decreases
indefinitely, a contradiction.
\end{proof}

Together with the rectangle witnesses for $a\le3$, the bent-strip
construction, and Proposition~\ref{prop:hierarchy}, these two theorems
prove Corollary~\ref{cor:gun}, including the rep-tile entry. The case
$n=0$ in that corollary is an L triomino and has a two-copy rectangle
construction from Section~\ref{sec:lpositive}.

\section{Boundary obstructions for T-polyominoes}
\label{obs:boundary}

Throughout this section let $T=\Pshape{a}{b}{c}{0}$, where
$a\ge c\ge1$ and $b\ge1$, and put $A=a+c+1$.  The \emph{crossbar}
has $A$ cells; the perpendicular \emph{stem} has $b$ cells beyond its
junction.  All coordinates below denote cells, rather than vertices.

\begin{lemma}[The quadrant corner]\label{obs:quadrant}
If $c\ge2$, then $T$ does not tile a quadrant.
\end{lemma}
\begin{proof}
A tile covering $(0,0)$ in $\N^2$ has a crossbar endpoint there:
a stem contact would put one crossbar arm outside the quadrant.
After diagonal reflection, its crossbar is $[0,A-1]\times\{0\}$ and
its upward stem has abscissa $r\in\{a,c\}$.

The tile covering $(0,1)$ cannot have a horizontal crossbar, since it
would meet that stem.  A horizontal stem contact would put a vertical
crossbar arm below row zero, since both such arms have length at least
two.  Thus its vertical crossbar must start at $(0,1)$ and its stem
point right.  Now $(1,1)$ has no cover.  A horizontal crossbar through
it must begin at $x=1$ and intersects the first stem; a vertical
crossbar must begin at $y=1$ and intersects the second stem.  Either
perpendicular-stem contact puts a crossbar arm beyond the quadrant.
\end{proof}

We use the following small-gap argument repeatedly.  The occupied
supports in its statement are unions of \emph{whole tiles} selected
from a hypothetical tiling.  A tile meeting an unoccupied cell is
therefore disjoint from all these supports.

\begin{lemma}[Blocked short run]\label{obs:short-run}
Suppose $b\le A-1$.  A run of $g$ consecutive cells in one row,
where $3\le g<A$, cannot be filled if its immediately lower neighbors
and its two immediate horizontal neighbors are occupied, while the
run is disjoint from the supporting tiles.
\end{lemma}
\begin{proof}
A horizontal crossbar cannot fit.  A horizontal stem contact either
crosses an occupied endpoint or puts its junction above the occupied
floor, where a downward crossbar arm collides with the floor.
Each cell consequently starts either a vertical crossbar, called $V$,
or an upward vertical stem whose junction is above it, called $S$.
Their vertical segments have respectively $A$ and $b+1$ cells.
An $S$ crossbar intersects its neighboring starter at height $b$
above the run.  Thus all starters must be $V$.  Two adjacent $V$
tiles can be disjoint only when their stems point away from one
another; three consecutive $V$ tiles are impossible.
\end{proof}

The same last argument shows that a consecutive block of $V/S$
starters has length at most two when $b\le A-1$, even if horizontal
crossbars occupy other parts of the row.

\begin{theorem}\label{obs:T-halfplane}
If $b\ge2$, then $T$ does not tile a half-plane.
\end{theorem}
\begin{proof}
Work in $y\ge0$.  A boundary tile has one of three forms: $H$, with
its horizontal crossbar on row zero and stem up; $V$, with its
vertical crossbar starting on row zero; or $S$, with the tip of its
vertical stem on row zero and horizontal crossbar at height $b$.

First, if $b>a$, a rightward boundary $V$ at $x=0$, with junction
height $r\in\{a,c\}$, is impossible.  At $(1,0)$ a $V$ or $S$
intersects its stem at $(1,r)$.  An $H$ must begin at $x=1$ and
has stem at $1+s$, $s\in\{a,c\}$; this intersects the old stem
because $1+s\le a+1\le b$ and $r<b$.  Reflect for a leftward $V$.

Next exclude $S$ for every $b\ge2$.  Put its tip at $(0,0)$ and
its length-$a$ crossbar arm to the right.  Its neighbor $(1,0)$
cannot be $S$.  It cannot be $V$: for $b\le A-1$ there is an
intersection at $(1,b)$, and otherwise the preceding exclusion
applies.  Hence this neighbor begins an $H$.  Avoiding the old
crossbar forces its junction to $x=a+1$.  The two tiles enclose
the nonempty rectangle
\[
       \{1,\ldots,a\}\times\{1,\ldots,b-1\}.
\]
Any tile meeting its interior must lie wholly inside, by connectivity
and the occupied boundary.  A horizontal crossbar is too wide;
a vertical crossbar would require both $b+1\le a$ and $A\le b-1$,
which are incompatible.

Assume first $c\ge2$.  For $b\ge A-1$, $V$ has already been
excluded.  Otherwise put a proposed rightward $V$ at $x=0$ and
its junction at height $r\in\{a,c\}$.  The cell $(1,0)$ must
begin an $H$ whose stem is at $j=1+s$, $s\in\{a,c\}$.
The unoccupied cell $(1,1)$ has no cover: a horizontal crossbar
hits the $H$ stem; a vertical crossbar hits the $V$ stem; a
vertical stem contact either overlaps row zero or puts a horizontal
arm across $x=0$ below height $b+2\le A$; a horizontal stem
contact puts a crossbar arm below row zero.  Thus only $H$ tiles
meet the boundary.

Their crossbars partition row zero.  If consecutive junction
offsets are $r_k,r_{k+1}\in\{a,c\}$, the gap between their stems
in row one has width
\begin{equation}\label{obs:gap-count}
             g_k=A-1+r_{k+1}-r_k\ge2c.
\end{equation}
Some consecutive pair satisfies $r_{k+1}\le r_k$, since an infinite
sequence in a two-element set cannot strictly increase.  Hence
$4\le g_k<A$.  If $b\le A-1$, Lemma~\ref{obs:short-run} applies.
If $b\ge A$, no horizontal crossbar fits and the cells start $V/S$
tiles.  A $V$ stem hits its neighbor, including an endpoint support
when it points outward: all the necessary vertical segments reach
its junction height, at most $a+1\le b$.  Only $S$ starters remain,
and adjacent ones overlap.

It remains to consider $c=1$, so $A=a+2$.  If $b>a$, the preceding
exclusions leave only $H$ boundary contacts and a gap with
$2\le g<A$.  For $b\ge A$ the same neighbor argument applies.
For $b=A-1$, any $S$ overlaps a neighbor; three $V$ tiles are
impossible, and when $g=2$ their outward stems hit the two
boundary stems.

Suppose $2\le b\le a$.  An admissible boundary $V$, up to
reflection and translation, has the forced neighboring pair
\[
\begin{array}{ll}
 V_0:& x=0,\ 0\le y\le a+1;\quad y=1,\ 1\le x\le b,\\
 H_0:& y=0,\ 1\le x\le a+2;\quad x=a+1,\ 1\le y\le b.
\end{array}
\]
Indeed, a $V$ with junction height $a$ makes $(1,1)$ impossible
by the four contact types used above; height one then forces
the neighboring $H$ stem to $a+1$.
For $a-b\ge3$, the row-one gap $b+1,\ldots,a$ contradicts
Lemma~\ref{obs:short-run}.  For $b\ge3$ and $a-b\in\{0,1,2\}$,
there is a blocked run of width $b$ in row two or three.
If $a=b$, use row two immediately.  If $a=b+1$, the unique
row-one gap starts $V$ or $S$: it either supplies the right wall
in row two, or is a low-junction leftward $V$ whose stem fills
row two, supplying the floor for row three.  If $a=b+2$, the
two row-one gaps force $V$ starters pointing outward.  The left
one gives these same two alternatives; the right one must have
a high junction, since a low rightward stem hits $H_0$.

For $a\ge7$ these arguments exclude every boundary $V$.
Equal consecutive $H$ junction offsets would give a forbidden
gap of length $a+1$; therefore offsets alternate between $1$ and
$a$, creating a gap of length $2a$.  At most one crossbar of
length $a+2$ fits there.  With none, three starters are impossible;
with one, its two end runs contain $a-2$ starters altogether,
at most four.  This gives $a\le6$, a contradiction.

The remaining fifteen cases are exactly
$2\le b\le a\le6$.  Proposition~\ref{obs:finite-boundary}
completes the proof.
\end{proof}

\begin{proposition}[Computer-assisted finite boundary cases]
\label{obs:finite-boundary}
For $2\le b\le a\le6$, no disjoint complete copies of
$\Pshape{a}{b}{1}{0}$ lying in $y\ge0$ cover the
boundary patch $\{0,\ldots,3(a+2)-1\}\times\{0,1,2\}$.
\end{proposition}
\begin{proof}
The fifteen certificates allow arbitrary overhang through the three
artificial sides of the patch.  Their acyclic case graphs have
693--1832 nodes.  Every translated, oriented tile through a queried
cell is included or excluded by an explicit intersection or
half-plane test; every branch ends without a possible cover.
Theorem~\ref{thm:certificate-soundness} applies.
The formal entry point is \nolinkurl{THalfPlaneFinite.finite_no_half_plane} in \nolinkurl{HPFinite.lean};
the certificate files are \nolinkurl{t_hpcert_a_b_1.json}, with the
parameters substituted.
\end{proof}

Together with the positive constructions, these obstructions give
half-plane tileability exactly when $b=1$, and quadrant tileability
exactly when $b=c=1$.

\section{Plane obstructions for T-polyominoes}
\label{obs:T-plane}

Only $b\ge3$ and $c\ge2$ remain after the plane constructions.
We retain $A=a+c+1$.

\begin{lemma}[Long-stem gaps]\label{obs:deep-run}
Assume $b\ge A$ and $c\ge2$.
\begin{enumerate}
\item A blocked run of length $2\le g<A$ is impossible if both
endpoint walls occupy the cells at heights $a$ and $c$ above
the run.  One such wall suffices when $g=3$; neither extra wall
is needed when $4\le g<A$.
\item A blocked run of any length $g\ge3$ is impossible if one
endpoint wall extends from row $h$ through row $h+1+a$, and its
supporting tile leaves the interior of row $h+1$ unoccupied.
As before, the run and its floor and endpoints satisfy the
whole-tile support convention.
\item If $c\ge3$, a run of length $b$ is impossible under the
weaker condition that one endpoint wall continues through
row $h+1$ and its tile leaves that row's run interior unoccupied.
\end{enumerate}
\end{lemma}
\begin{proof}
In a short run only $V/S$ starters are possible.  A $V$ stem
intersects any neighboring starter in its direction, so $V$ can
occur only at an end and point outward.  A tall endpoint wall
excludes that possibility.  The stated lengths then force two
adjacent $S$ tiles, whose crossbars intersect.

For the second assertion, suppose a horizontal crossbar occurs
and select the leftmost.  It points its stem upward.  Above its
left arm is a run of length $r\in\{a,c\}$, with its bar as floor
and its upward stem as right wall.  The left wall is either the
given wall or the preceding $V/S$ starter.  Such a starter has
no horizontal cells one row above its bottom, since $c\ge2$;
its vertical segment reaches the required heights.  Both new
walls therefore satisfy the first assertion, which excludes this
crossbar.  Without crossbars the original run consists of
starters.  For $g\ge4$ its two interior positions must be $S$;
for $g=3$ the deep wall excludes $V$ at its adjacent end.
Again two adjacent $S$ tiles result.
For the third assertion use the same leftmost-crossbar argument.
The new short run has length $r\ge3$; its upward stem supplies
the single tall wall needed when $r=3$, and larger $r$ needs
none.  Only the additional endpoint cell in row $h+1$ is
therefore required.  Without crossbars, the length-$b$ run
has adjacent interior $S$ starters since $b\ge A\ge7$.
\end{proof}

\begin{theorem}\label{obs:three-long-T}
If $a\ge c\ge3$ and $b\ge3$, then $T$ does not tile the plane.
\end{theorem}
\begin{proof}
First suppose $b\le A-1$.  Place an upright tile $O$ with
junction $(0,0)$.  At $(1,1)$ an existing cell immediately west
and immediately south forces the covering tile to have an
arm endpoint there.  An upward-stem horizontal crossbar creates
a blocked row of length $r\in\{a,c\}$ above its left arm.
A rightward-stem vertical crossbar creates the rotated blocked
run of length $r$.  A perpendicular-stem tip creates a blocked
run of length $b$, horizontally or vertically.  All are excluded
by Lemma~\ref{obs:short-run}.  Thus the only possibilities are a
horizontal crossbar with stem down or a vertical crossbar with
stem left, with their junctions far enough away to avoid $O$.

The same alternatives hold at $(-1,1)$.  Both corners cannot
use vertical crossbars: their columns are two units apart, and
the inward stem at the lower junction intersects the other bar.
Reflecting if necessary, obtain a downward-stem tile $U$ whose
crossbar is $[-A,-1]\times\{1\}$.

Now cover $p=(-1,2)$.  Its occupied lower and right neighbors
again force an endpoint contact.  A horizontal crossbar either
hits $U$ with its downward stem or creates a blocked row of
length $r$ with an upward stem.  A horizontal perpendicular
stem puts its junction above $U$ and hits it.  A vertical
perpendicular stem creates a rotated blocked run of length $b$.
A leftward-stem vertical crossbar creates a rotated run of
length $r$.  Consequently $p$ starts a vertical crossbar $V$
with stem right, occupying $x=-1$, $2\le y\le A+1$.

At $q=(-2,2)$ the same cases exclude horizontal crossbars and
leftward-stem vertical crossbars.  A rightward stem or vertical
perpendicular-stem cover intersects $V$.  A horizontal
perpendicular stem intersects $U$ unless $b=A-1$.  In that last
case it has junction $(-A-1,2)$ and fills row two from $-A$
through $-2$, leaving above itself a blocked run of length
$A-1$ between its crossbar and $V$.  This is also impossible.

For $b\ge A$, use Lemma~\ref{obs:deep-run} instead.  The
away-pointing crossbar cases create runs of length $a$ or $c$;
the required deep wall is supplied by a length-$b$ stem.
The perpendicular-stem cases create runs of length $b$, with
the same stem as floor and the adjacent bars as straight
endpoint supports.  The third part of the lemma applies,
including its requirement one row beyond the run.  Thus the
antiparallel pair $O,U$ and the forced crossbar $V$ at $p$ are
obtained as above.  Its rightward stem already hits $O$:
its junction height is $2+r\le a+2<A\le b$.
\end{proof}

For $c=2$ the exceptional equality $b=a+3$ has a plane
construction.  The ranges on either side require different arguments.

\begin{proposition}[Computer-assisted short-stem obstruction]
\label{obs:two-short}
If $a\ge2$ and $3\le b<a+3$, then
$\Pshape{a}{b}{2}{0}$ does not tile the plane.
\end{proposition}
\begin{proof}
The symbolic certificate
\nolinkurl{t_plane_negative_corners_two_short_d10.json}
has 300 nodes and 208 terminal contradictions, and adds at most
seven tiles to an anchored initial tile.  Its parameters satisfy
only the displayed inequalities.  At each node all eight
orientations and an arbitrary integer junction of a tile through
the selected cell are tested.  The resulting exhaustive
arithmetic alternatives and the geometric replay are checked
in Lean.  Normalize a tile in a hypothetical plane tiling to
the anchor and apply Theorem~\ref{thm:certificate-soundness}.
The unanchored formal theorem is
\nolinkurl{short_stem_two_arm_no_plane}.
\end{proof}

\begin{theorem}\label{obs:two-long}
If $a\ge2$ and $b>a+3$, then $\Pshape{a}{b}{2}{0}$ does not
tile the plane.
\end{theorem}
\begin{proof}
Here $A=a+3$ and $b\ge A+1$.  We describe placements by their
junction followed by their east, north, west, and south arms.
Apply Lemma~\ref{obs:deep-run} at concave corners.  An
antiparallel pair, with $O$ upright at zero and a downward-stem
bar $U$ in $[-A,-1]\times\{1\}$, is impossible: at $(-1,2)$
horizontal bars make short deep-wall gaps, a rightward vertical
bar hits $O$, and a leftward vertical bar makes a rotated
short gap.  Perpendicular-stem contacts make length-$b$ gaps;
the original stem or the bar of $U$ supplies the deep wall.

At a concave corner of an upright tile the same tests leave
only an outward vertical crossbar with lower arm two, or a
vertical perpendicular-stem tip.  The two concave corners
cannot both use perpendicular-stem tips: their horizontal
crossbars have junctions two units apart at the same height.
Thus, after reflection, the corner $(1,1)$ is covered by
\[
                 U=(1,3;\ b,a,0,2).
\]
Covering $(2,1)$ and $(2,2)$ now forces respectively a
downward-stem horizontal bar $R$ and a leftward perpendicular
stem $W$.  Indeed, after transposition they are two starters;
two $S$ tiles overlap, inward $V$ stems intersect, and the upper
wall excludes an outward $V$ at the upper cell.  Thus
\[
\begin{split}
 R\text{ has bar }&[2,A+1]\times\{1\},\\
 W\text{ has junction }&(b+2,2)\text{ and left stem of length }b.
\end{split}
\]
Either vertical arm order of $W$ is allowed.  Its upper-left
concave corner is covered by the stem tip of $U$, so its
lower-left corner must have the other admissible cover.  This
forces
\[
                 Z=(b-1,1;\ 2,0,a,b).
\]

Consider the $b$ cells $[2,b+1]\times\{1\}$ below the stem
of $W$.  Their outside endpoints belong to $U,W$, and their
two end cells belong to the distinct downward bars $R,Z$.
If those bars overlap, we are done.  Every other boundary
contact is, downward, an $H$, $V$, or $S$.  An interior $V$
intersects its neighbor in its stem direction: a $V/S$ neighbor
has a sufficiently long vertical segment, while an $H$
neighbor puts its stem at distance at most $a+1<b$.
An interior $S$ must therefore have an $H$ neighbor in the
direction of its length-$a$ arm.  Avoiding overlap forces that
neighbor's junction to distance $a+1$ and encloses an
$a$-by-$(b-1)$ rectangle.  No orientation fits: the two
possible widths are $A>a$ and $b+1>a$.

Hence the interval consists of consecutive downward $H$ bars,
and contains at least two since $b>A$.  Consecutive junction
offsets $r,s\in\{2,a\}$ enclose, immediately below their bars,
a run of width $A-1+s-r\ge4$.  Their length-$b$ stems provide
the straight deep walls.  Lemma~\ref{obs:deep-run}, applied
downward, gives the final contradiction.
\end{proof}

\section{Obstructions for genuine crosses}
\label{obs:crosses}

\begin{lemma}\label{obs:cross-halfplane}
No family of crosses with four positive arms tiles a half-plane,
even if the crosses have different arm lengths.
\end{lemma}
\begin{proof}
A boundary contact must be the tip of a downward arm.  The
tiles covering $(0,0)$ and $(1,0)$ have junctions $(0,r)$ and
$(1,s)$, with $r,s\ge1$.  If $r\le s$, the first tile's right
arm intersects the second tile's vertical arm at $(1,r)$.
If $s\le r$, reverse the roles.  Either way the tiles overlap.
\end{proof}

\begin{theorem}\label{obs:long-cross}
If $a,b,c,d\ge2$, then $\Pshape{a}{b}{c}{d}$ does not tile
the plane.
\end{theorem}
\begin{proof}
Let $M$ be the longest arm.  Normalize a tile $O$ so its
junction is zero and a longest arm points east.  At each of
$(\pm1,\pm1)$, the covering tile ends an arm.  Assign the
diagonal cell to the cardinal direction from which this arm
arrives.  Two diagonal cells cannot be assigned the same
direction: for example, junctions $(r,1),(s,-1)$ supplying
both eastern cells would overlap, because the lower-reaching
arm at the nearer junction meets the farther horizontal arm.
All arms have length at least two, as required for this test.
Thus all four directions occur exactly once.

Reflect vertically so the northeast diagonal is supplied from
the east.  Its covering tile has junction $(r,1)$ with $r>M$,
and west arm ending at $(1,1)$.  Since that arm has length at
most $M$, necessarily $r=M+1$.  Its horizontal bar supplies a
floor below $(1,2)$ and $(2,2)$; its vertical bar and that of
$O$ supply walls at $x=M+1$ and $x=0$.
Neither cell can lie on a new horizontal bar: a junction
between the walls sends an arm into the floor, while a
junction outside sends its bar through a wall.
Both must therefore belong to distinct vertical arms arriving
from above.  At the lower of their two junction heights, a
horizontal arm intersects the other vertical arm.  Contradiction.
\end{proof}

\begin{figure}[htbp]\centering
\includegraphics[width=.60\linewidth]{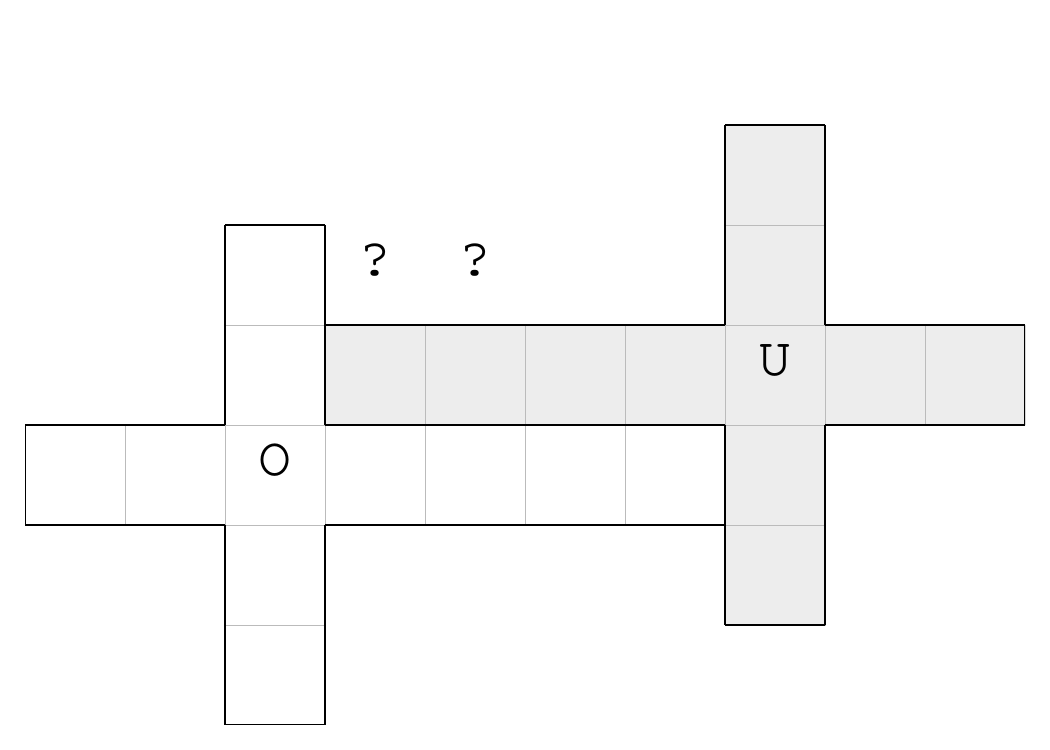}
\caption{The trapped adjacent cells in the all-long-arm cross argument.
The two supporting tiles force vertical arrivals from above, which
must intersect. This local picture illustrates Theorem~\ref{obs:long-cross}.}
\end{figure}

\begin{proposition}[Computer-assisted unit-arm cross obstructions]
\label{obs:unit-cross}
Neither of the following families tiles the plane:
\[
 \Pshape{a}{1}{c}{d}\quad(2\le a\le c,\ d\ge2),
 \qquad
 \Pshape{1}{1}{c}{d}\quad(2\le c\le d).
\]
\end{proposition}
\begin{proof}
The two symbolic trees have respectively 89 and 57 nodes.
Each starts with the prototile at zero and selects $(-1,-1)$.
In the adjacent-unit case there are eight first placements;
in the one-unit case there are twelve.  Each first placement
requires one further tile, after which a selected cell has no
disjoint cover.  For example, an adjacent-unit branch places
\[
 (-1,-1-c;\ 1,c,d,1),\qquad
 (-2-c,-c;\ c,d,1,1),
\]
and then queries $(-c,1-c)$, which cannot be covered.
All alternatives, including arbitrary integer junctions, are
checked symbolically under the displayed unbounded parameter
conditions.  Theorem~\ref{thm:certificate-soundness} gives the
obstructions.  Translation and dihedral normalization are
included in the formal wrappers
\nolinkurl{CrossClassification.one_no_plane} and
\nolinkurl{CrossClassification.adjacent_no_plane}.
\end{proof}

Every remaining positive-arm cross has an opposite pair of
unit arms.  The construction for that case, together with
Lemma~\ref{obs:cross-halfplane}, completes the cross classification.

\section{The earlier L classification and its formal reconstruction}
\label{obs:L}

For $L=\Pshape{a}{b}{0}{0}$ normalize $a\ge b\ge1$.
Its bounding dimensions are $(a+1)$ by $(b+1)$.
The classification is the author's earlier result
\cite{raychevQuadrant2021,raychevHalfPlane2021}: $b=1$ admits
a rectangle; $b=2$ and $(a,b)=(4,3)$ admit a strip but no
quadrant; all other cases admit the plane but no half-plane.
Here we summarize the negative proofs and specify where the
formal reconstruction uses computer-assisted local arguments.

The quadrant paper excludes every $a\ge b\ge2$ by local
corner arguments.  Although one theorem is phrased for a
rectangle, its proof uses only the two adjacent boundary rays;
the half-plane paper explicitly records the quadrant conclusion.
The cases are separated according to equal legs, the two
small $b=2$ cases, the remaining $b=2$ family, and $a>b\ge3$.
Endpoint contacts are excluded first; the remaining configurations
force an uncovered cell or an enclosed rectangle too short for
either orientation.  These arguments are reproduced in the
formal classification, rather than assumed as axioms.

\begin{proposition}[Reconstructed L boundary reductions]
\label{obs:L-reductions}
Let $a>b\ge3$ and $a\ge5$.  In a hypothetical half-plane
tiling by $L$, the following consequences hold:
\begin{enumerate}
\item No boundary contact is solely a vertical tip, and no
horizontal boundary leg is the short leg.
\item Consecutive long boundary legs have alternating stem
orientations.  Eight consecutive such legs can be extracted
without assuming any periodicity.
\item Except for the separately excluded case $(a,b)=(5,3)$,
four central legs may be normalized to junctions
$-1,0,2a+1,2a+2$ in row zero.  Coverage above them forces
$a=b+2$ and, up to reflection, the two row-one tiles
\[
       (1,1;\ a,b,0,0),\qquad (2a,1;\ 0,a,b,0).
\]
\end{enumerate}
\end{proposition}
\begin{proof}[Proof structure and computer-assisted components]
The first two steps follow the boundary strategy in
\cite{raychevHalfPlane2021}.  Once tips are excluded, normalize
a proposed short boundary leg to $(0,0; b,a,0,0)$.
There are seven possible disjoint covers of $(1,1)$.
Six have finite obstructions.  The seventh is eliminated by
first applying the six established consequences to the tile
covering the cell just left of the boundary leg, and then
checking its remaining long-leg alternatives.  This dependency
order avoids circular use of the short-leg exclusion.
Equal-facing consecutive long legs similarly have a finite
obstruction, proving alternation.

Outside that exceptional case, the central stems bound the
first-row interval $1,\ldots,2a$.
Tip and corner exclusions force its coverage by horizontal
legs.  Successor and endpoint counts, together with the
outward-short-leg obstructions for $a=b+1$ and $a\ge b+3$,
leave $a=b+2$.  At that difference the endpoint count gives
the displayed mixed pair or its reflection.  These are
symbolic reductions for unbounded arm lengths, replayed by
\nolinkurl{LHalfPlaneShortbarAssembly}, \nolinkurl{LHalfPlaneAssembly},
and \nolinkurl{LHalfPlaneAlternating}; their finite branches use
Theorem~\ref{thm:certificate-soundness}.
\end{proof}

\begin{figure}[htbp]\centering
\includegraphics[width=\linewidth]{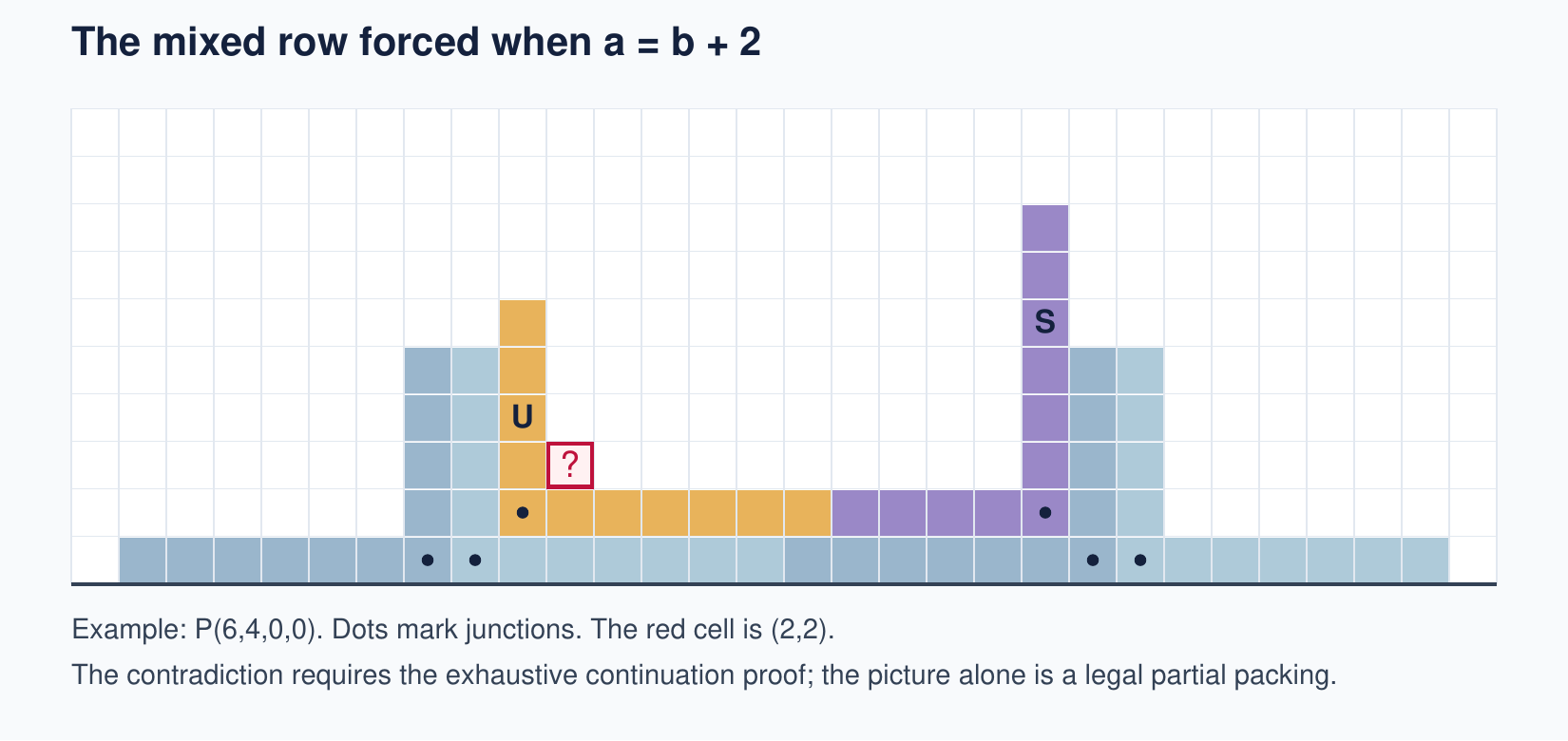}
\caption{The mixed configuration in the final unequal-L obstruction,
illustrated for $a=6$, $b=4$. Four boundary tiles support the mixed
row; the certificate starts with the uncovered cell $(2,2)$.}
\end{figure}

\begin{proposition}[Computer-assisted final L obstructions]
\label{obs:L-final}
The mixed configuration of Proposition~\ref{obs:L-reductions}
is impossible for $a=b+2$, $b\ge4$.  The case $(a,b)=(5,3)$
also has no half-plane tiling.
\end{proposition}
\begin{proof}
The mixed configuration uses the two specified row-one tiles
and four supporting boundary tiles.  Starting at $(2,2)$,
its 522-state certificate checks every possible covering
placement and closes every branch.  The exceptional $(5,3)$
case has separate short-bar and alternating-eight-bar
certificates with 421 and 539 states.  These certificates are
replayed against arbitrary whole-tile tilings; no bounded-width
or periodic-tiling hypothesis is imposed.

Some leaves close with an immediate impossible cover; others
use a short blocked run of length $g\le b$.  Five starters
are impossible without additional caps.  Four are excluded
by one endpoint cap at height $b$; three by both such caps;
two by both height-$b$ caps and one height-$a$ cap; even a
single starter is excluded when both endpoints are capped
at both heights.  The supporting cells belong to already
selected whole tiles, and the reflected and rotated versions
check that all queries remain inside the half-plane.
Theorem~\ref{thm:certificate-soundness} and these local lemmas
yield the claims.  The formal conclusions are
\nolinkurl{LHalfPlaneMixedMotif.mixed_impossible} and
\nolinkurl{LHalfPlaneFiveThree.no_halfplane}.
\end{proof}

Together with the equal-leg obstruction for $a=b\ge3$ from
\cite{raychevHalfPlane2021}, formalized in
\nolinkurl{LHalfPlaneEqual}, the two propositions give the entire
negative half-plane classification.  In particular, this is a
formal reconstruction of the author's prior L theorem, not a
new attribution of that classification.  The mixed certificate
replaces the paper's final diagram argument; it is not a
line-by-line formal translation of that argument.

\section{Completion of the classification}\label{sec:assembly}
\begin{proof}[Proof of Theorem~\ref{thm:classification}]
We first account for every tuple.  With no positive arms, or one
positive arm, the shape is a bar.  Two opposite positive arms also
give a bar, of length their sum plus one.  Two adjacent positive
arms give an L, which can be normalized to $a\ge b\ge1$.
Three positive arms give a T: rotate its missing arm south and
reflect horizontally if necessary to obtain $a\ge c\ge1$.
Four positive arms give a genuine cross.  Integer translations
and square-grid symmetries preserve all eight capabilities,
including integer-scale rep-tilings.  This reduction covers
degenerate tuples and does not require the parametrization to
be unique.

Bars, the L shapes with $b=1$, and the three small T shapes
with $b=c=1$, $a\le3$, have rectangle constructions in
Section~\ref{sec:constructions}.  Proposition~\ref{prop:hierarchy}
therefore gives every capability in the rectangle profile.

For the remaining L shapes, Section~\ref{sec:lpositive} gives
strip tilings when $b=2$ or $(a,b)=(4,3)$ and plane tilings
in every case.  Section~\ref{obs:L} excludes a quadrant for
all $b\ge2$, and a half-plane outside those strip cases.
Thus the former have exactly the strip profile and the latter
exactly the plane-only profile.  In particular, absence of a
quadrant excludes $\Rep$ as well as $\BS$, $\HS$, and $\R$.

For a T with $b=1$, Proposition~\ref{prop:tstrip} always
gives a strip.  If $c\ge2$, Lemma~\ref{obs:quadrant}
excludes a quadrant, hence gives exactly the strip profile.
If $c=1$ and $a\ge4$, the same construction gives a bent
strip, while Theorems~\ref{thm:gunhs} and~\ref{thm:gunrep}
exclude a half-strip, a rectangle, and a rep-tiling.
The hierarchy supplies the other four positive capabilities,
giving exactly the bent-strip profile.

For a T with $b\ge2$, Theorem~\ref{obs:T-halfplane}
excludes a half-plane and hence every listed capability
except possibly the plane.  Proposition~\ref{prop:tplane}
gives a plane tiling when $c=1$, $b=2$, or
$(c=2,\ b=a+3)$.  In the complement, $b\ge3$ and $c\ge2$.
When $c\ge3$, Theorem~\ref{obs:three-long-T} excludes the
plane; when $c=2$, Proposition~\ref{obs:two-short} and
Theorem~\ref{obs:two-long} cover respectively $b<a+3$ and
$b>a+3$.  Thus this complement consists precisely of non-tilers.

Finally, no genuine cross tiles a half-plane, by
Lemma~\ref{obs:cross-halfplane}.  A cross with opposite unit
arms has a plane tiling by Proposition~\ref{prop:crosspositive},
after rotation if necessary.  If no opposite pair is unit,
there are either no unit arms, exactly one, or exactly two
adjacent unit arms.  Theorem~\ref{obs:long-cross} excludes the
first case, and Proposition~\ref{obs:unit-cross}, after
dihedral normalization, excludes the other two.  This proves
both cross rows and completes every entry of the table.
\end{proof}

\section{Formal verification and reproducibility}
\label{sec:formalization}

Theorem~\ref{thm:classification} has a complete formal proof in Lean
4.33.1~\cite{demouraUllrich2021}, using Lean's standard library. The source package contains 538 local
modules and a single unconditional entry point:
\begin{verbatim}
theorem all_tuples_classified (a b c d : Nat) :
    Classified a b c d (classifyTuple a b c d)
\end{verbatim}
This declaration belongs to the namespace \nolinkurl{PolyominoFormal} in
\nolinkurl{MainClassification.lean}. Its only arguments are the four arm
lengths. In particular, it has no unproved local-obstruction or
family-classification hypothesis.

\paragraph{What is formalized.}
A lattice region is a predicate on \(\Z^2\). A placed tile is represented
by its integer junction coordinates and four directed arm lengths.
\nolinkurl{Copy} lists the eight rotations and reflections of the specified
shape. \nolinkurl{Tiling} consists of an arbitrary predicate on placed tiles,
together with proofs of congruence, containment of every cell of each whole
tile in the region, coverage of every region cell, and pairwise cellwise
disjointness. Thus an infinite tiling need not be periodic, computable, or
described by a finite motif. For finite target regions this definition
also gives ordinary finite tilings: every tile contains its junction cell,
and distinct tiles have distinct junction cells inside the target.

The region predicates define the full rectangle, half-strip, bent strip,
quadrant, strip, half-plane, and plane. The four predicates involving
widths or heights existentially quantify over all positive integer
dimensions. Consequently, a half-strip impossibility theorem excludes
every width, rather than a finite range examined by a search. The
rep-tiling predicate quantifies over natural linear scales \(k\geq2\).
Its target replaces each original cell by a \(k\times k\) block; it is
not obtained by multiplying the arm lengths while retaining unit
thickness. The formal claim concerns congruent original tiles, integer
translations, and lattice rotations and reflections, with this integer
scale convention. It makes no assertion about arbitrary off-grid
Euclidean dissections.

The predicate \nolinkurl{Classified} states an equivalence for \emph{each} of
the eight capabilities. It records both positive and negative results,
not merely a strongest known construction. Five distinct capability
profiles occur in this family. \nolinkurl{TupleNormalization.lean} handles
rotations, reflections, bars, and zero-arm cases, including
\(\Pshape{0}{0}{0}{0}\). It reduces genuine T shapes to
\(a\geq c\geq1\), \(b\geq1\), and L shapes to sorted adjacent arms.
The symmetry arguments transport the enlarged target as well as the
small tiles when proving invariance of rep-tiling.

\paragraph{Finite certificates and infinite tilings.}
The computer-assisted obstructions have two logically separate parts.
A finite case analysis proves a local statement; a geometric argument
extracts its hypotheses from an arbitrary exact tiling. Covering a chosen
uncovered cell supplies an actual tile from that tiling. Completeness
lemmas enumerate every possible orientation and relative placement that
could cover it. Each resulting branch proves an overlap, a containment
violation, another impossible configuration, or a stated successor
configuration. The arithmetic obligations are proved in Lean. External
search programs select useful cells and generate candidate case trees,
but their answers are not assumed by the formal proof.

For \(\Pshape{n}{1}{1}{0}\), the finite corner analysis is packaged
separately for \(n=4\), \(n=5\), and the symbolic range \(n\geq6\).
These cases feed the same arbitrary-width half-strip argument. If a
corner at height \(Y\) in a half-strip of height \(H\) has type \(s\),
its rank is
\[
                    3(H-Y)+r(s),\qquad 0\leq r(s)\leq2.
\]
Every permitted transition decreases this nonnegative integer: a
transition either rises or decreases the auxiliary rank at unchanged
height. This is a well-founded descent, not an inference from checking
large finite rectangles. The L obstructions likewise connect finite
placement arguments to actual half-plane tilings through separately
proved boundary reductions. Other formal components construct infinite
tilings explicitly and prove the compactness and hierarchy implications
used in assembling the full profiles.

\begin{table}[tb]
\centering
\small
\begin{tabular}{p{0.29\linewidth}p{0.65\linewidth}}
\hline
Mathematical result & Lean declaration and source module\\
\hline
Complete classification &
\nolinkurl{all_tuples_classified},
\nolinkurl{MainClassification.lean}\\[3pt]
Sorted L classification &
\nolinkurl{sorted_L_classified},
\nolinkurl{MainClassification.lean}\\[3pt]
Unequal L obstruction &
\nolinkurl{LHalfPlaneUnequal.no_halfplane},
\nolinkurl{LHalfPlaneUnequal.lean}\\[3pt]
Exceptional L \(\Pshape{5}{3}{0}{0}\) &
\nolinkurl{LHalfPlaneFiveThree.no_halfplane},
\nolinkurl{LHalfPlaneFiveThree.lean}\\[3pt]
Full T classification &
\nolinkurl{normalized_T_classified},
\nolinkurl{TClassification.lean}\\[3pt]
\(\Pshape{4}{1}{1}{0}\), \(\Pshape{5}{1}{1}{0}\) &
\nolinkurl{p4_exact_bs}, \nolinkurl{p5_exact_bs},
\nolinkurl{InitialTwoClassification.lean}\\[3pt]
All \(\Pshape{n}{1}{1}{0}\) &
\nolinkurl{classify_gun},
\nolinkurl{GunFamilyClassification.lean}\\[3pt]
Four positive arms &
\nolinkurl{classify_positive_cross},
\nolinkurl{PositiveCrossClassification.lean}\\
\hline
\end{tabular}
\caption{Principal formal statements. All declarations in this table
have the namespace prefix \nolinkurl{PolyominoFormal.}; files are relative
to the archive's \nolinkurl{lean/} directory.}
\label{tab:lean-map}
\end{table}

\paragraph{Trust and reproduction.}
The accepted build verified the exact source dependency closure of
\nolinkurl{MainClassification} and \nolinkurl{GunFamilyClassification}. Lean's
axiom report for \nolinkurl{all_tuples_classified} lists exactly
\nolinkurl{propext}, \nolinkurl{Classical.choice}, and \nolinkurl{Quot.sound},
the standard foundational axioms used by this development. The accepted
sources contain no admitted proofs, additional axioms, calls to
\nolinkurl{native_decide}, or trusted external solver answers. Certificate
checking produces ordinary Lean proof terms. This assurance concerns the
formal statements and their definitions; correspondence with the
mathematical exposition remains a separate matter for review.

The \href{\artifacturl}{frozen source archive} is identified by an immutable
repository commit. It contains the pinned toolchain, all 538 local source
modules, the build driver, a file-hash manifest, build receipts, and the
axiom reports. No precompiled local proof objects or external solver is
needed. After installing Lean 4.33.1, extracting the archive, and entering
its \nolinkurl{polyominoes} directory, run
\noindent\begin{minipage}{\linewidth}
\begin{verbatim}
python3 src/check_lean.py MainClassification \
  GunFamilyClassification --jobs 1
\end{verbatim}
\end{minipage}
The driver uses one worker and one Lean thread by default. It records
source, dependency, compiler-version, and compiled-object hashes in an
isolated output directory. An existing result is reused only when these
identifiers match; specifying a new \nolinkurl{--output} directory requests
a fresh build. The receipts document the completed verification;
rebuilding independently checks the proofs. The broader research archive
also preserves exploratory and abandoned attempts, and is not the
accepted proof dependency closure.

\section*{Acknowledgments}
\addcontentsline{toc}{section}{Acknowledgments}
I am deeply grateful to Professor Stanislav Harizanov, my mentor during the original
research, for recognizing my potential, offering important advice, and
carefully checking my work. I also thank my mathematics teacher Dimitar
Dimitrov, whose guidance extended far beyond the classroom. As a mentor
in life, he helped me find direction, encouraged me to pursue mathematical
research, and introduced me to Professor Harizanov. Their encouragement and support
were central to the work I undertook between October 2020 and March 2021.

The mathematical and computational role of Astra 6, operating through
the Codex harness, is described in Section~\ref{sec:intro}. The contributions
of earlier authors to the tiling results and constructions are credited
at their points of use.

\newcommand{\doi}[1]{doi:\ \href{https://doi.org/#1}{\nolinkurl{#1}}}
\bibliographystyle{plainnat}
\bibliography{references}
\end{document}